\documentclass[11pt,reqno]{amsart}

\usepackage[useregional,showseconds=false,showzone=false]{datetime2}

\usepackage{url}

\usepackage{palatino}
\usepackage[mathcal]{euler}

\usepackage{dsfont}
\usepackage{amsmath,amsthm,amssymb}
\allowdisplaybreaks

\usepackage[all]{xy}

\usepackage{amsfonts}
\usepackage{enumerate}

\usepackage[hidelinks]{hyperref}

\numberwithin{equation}{section}

 \swapnumbers

\theoremstyle{plain}
\newtheorem*{theorem*}{Theorem} 
\newtheorem*{lemma*}{Lemma}
\newtheorem*{assumption*}{Assumption}

\newtheorem{theorem}{Theorem} 
\newtheorem{lemma}[theorem]{Lemma}
\newtheorem{corollary}[theorem]{Corollary}
\newtheorem{proposition}[theorem]{Proposition}

\theoremstyle{definition}
\newtheorem{definition}[theorem]{Definition}

\newtheorem{remark}[theorem]{Remark}

\newtheorem*{remark*}{Remark}
\newtheorem*{remarks*}{Remarks}
\newtheorem*{observation*}{Observation}

\theoremstyle{remark}

\numberwithin{theorem}{section}

\theoremstyle{plain}
    \newtheorem*{thm*}{Theorem} 
        \newtheorem{thm}[theorem]{Theorem} 
    \newtheorem{lem}[theorem]{Lemma}

\theoremstyle{definition}
    \newtheorem*{defn*}{Definition}
    
    \newtheorem*{rmk*}{Remark}

\theoremstyle{remark}
    
\numberwithin{equation}{section}

\newcommand{\bC}{\mathbb{C}}

\newcommand{\DMC}{{\small \textsf{DMC}}}
\newcommand{\MC}{{\small \textsf{MC}}}
\newcommand{\DCone}{D}
\newcommand{\Susp}{{\small \textsf{S}}}

 \newcommand{\fat}[1]{n}

\DeclareMathOperator{\res}{res}

\renewcommand{\Re}{\operatorname{Re}}

\newcommand{\R}{\mathbb{R}}
\newcommand{\C}{\mathbb{C}}

\DeclareMathOperator{\Spec}{Spec}

\begin{document}

\title[Novodvorskii's Theorem and the Oka Principle]{On  Novodvorskii's Theorem\\ and the  Oka Principle} 
\author{Jacob Bradd}
\author{Nigel Higson}
\address{Department of Mathematics, Penn State University, University Park, PA 16802, USA}

\date{\today}

\begin{abstract}
We  give an exposition of Novodvorskii's theorem in Banach algebra K-theory, asserting that the Gelfand transform for a commutative Banach algebra induces an isomorphism in topological K-theory.
\end{abstract}

\maketitle

\nocite{Taylor75NATO}
  
\section{Introduction}
\label{sec-introduction}

The purpose of this paper is to give an expository account of the following theorem of Novodvorskii \cite{Novodvorski67}, which was one  of the most striking discoveries from the early days of $K$-theory for Banach algebras:

\begin{theorem*}
If $A$ is a commutative and unital Banach algebra $A$ with Gelfand spectrum $X$, then the Gelfand transform induces an isomorphism 
\[K_*(A) \stackrel \cong \longrightarrow K_*(C(X))
\]
in \textup{(}topological\textup{)} $K$-theory.
\end{theorem*}

There are several precursors to Novodvorskii's theorem. In fact  essentially the same result was proved by Arens \cite{Arens66}, in the same way, although $K$-theory was not mentioned. We refer the reader to \cite{Taylor75ICM,Taylor76} for  early surveys, as well as to the introduction of \cite{Bost90} for further remarks.  But the clean and simple statement of Novodvorskii's theorem causes it to stand out, especially now that  $K$-theory has become a familiar topic in Banach algebra theory.  

The theorem is a $K$-theoretic version of  the Oka principle from several complex variables, and it  may be proved by a reduction to the  Oka principle. This was Novodvorskii's approach. In contrast, the  proof  we shall present here     takes full  advantage  of the computational framework that $K$-theory provides, and minimizes prerequisites from several complex variables.  

The $K$-theory approach makes plain a striking analogy between the proof of the Oka principle and the  proof  of a classic result in topology, namely the Jordan-Brouwer separation theorem:

\begin{theorem*}
The complement of an embedded $(n{-}1)$-sphere in an $n$-sphere has precisely two connected components.
\end{theorem*}

The separation theorem   quickly  reduces to the following assertion, which is the heart of the matter:

\begin{theorem*} 
The complement  of an   embedded $k$-cube in an $n$-sphere has the singular homology of a point.
\end{theorem*}

This is proved by  induction on the dimension of the cube.  First, the dimension zero case is trivial.  Next, by writing a $k$-cube $C$ as a union of two closed half-cubes that intersect along a midplane,    by assuming the result for this midplane, which is a lower-dimen\-sional cube,  and by invoking the Mayer-Vietoris sequence, we find that the theorem is true for a given embedding of  $C$ if and only if it is true for the embeddings of  the two half-cubes.  

Each of these half-cubes we may in turn cut in two,   along  hyperplanes parallel to the first cut; then we may do the same to   the four resulting quarter cubes; and so on. A simple diagram chase and an application of the continuity property of singular homology  show that the theorem holds for the embedding of $C$ if and only if it holds for the embeddings of all 
$\cap C_j$, where $C_0{=}C$ and 
where  $C_j$ is one of the halves that is obtained by bisecting   $C_{j-1}$, as above. But $\cap _j C_j$ is a cube of lower dimension, so the induction hypothesis applies, and the proof is complete.

Novodvorskii's theorem  may be proved in essentially the same way, using a combination of Mayer-Vietoris and continuity in $K$-theory (meaning compatibility with direct limits).\footnote{See also \cite{Larusson03}, where essentially the same point about the role of the Mayer-Vietoris property in the proof of the Oka principle is made in more categorical language. We thank Jonathan Block for pointing this out to us.} 

To begin, the key instance of the theorem, to which all others may be reduced, is that of the Banach algebra $B(X)$ that is associated to a polynomially convex compact set $X \subseteq \C^n$, and is constructed  by completing the algebra of polynomial functions on $X$ in the uniform norm.  The Gelfand transform in this case is the inclusion of $B(X)$ into the algebra $C(X)$ of all continuous functions.  

The main task is to establish a $K$-theory Mayer-Vietoris sequence for the $B(X)$-algebras, leading to commuting diagrams of the form
\[
\xymatrix@C=12pt{
\cdots\ar[r] &  K_0(B(X_1\cap X_2))\ar[r]\ar[d] & K_1(B(X_1\cup X_2)) \ar[r]\ar[d] & K_1(B(X_1)\oplus B(X_2))\ar[d]\ar[r]& \cdots\\
\cdots\ar[r] &  K_0(C(X_1\cap X_2))\ar[r]  & K_1(C(X_1\cup X_2)) \ar[r] &K_1(C(X_1)\oplus C(X_2))\ar[r]& \cdots \\}
\]
Complex analysis enters here.  Then, with the Mayer-Vietoris sequence in hand, the proof of Novodvorskii's theorem is by induction on the dimension of $X$, using repeated divisions of $X$ by hyperplanes and an eventual direct limit argument.

More than once the Oka principle has been suggested as  a possible starting point for  new approaches to the Baum-Connes conjecture in the $K$-theory of group $C^*$-algebras \cite{BCH93}; see for example  \cite{AparicioEtAl19}.  Those ideas have  not yet advanced very far, but it was with  possible noncommutative generalizations   in mind that we  have written this exposition.

\section{The Mayer-Vietoris Property}

In this section we shall  gather some   results about Mayer-Vietoris sequences in Banach algebra $K$-theory.  Most are  easily derived from foundational properties of $K$-theory such as the six-term exact sequence, and so on.
Generally we shall simply state these results.
An exception is   Theorem~\ref{thm-improved-m-v-thm}, which is    more substantial and   less well known. 

We begin by recalling some facts about mapping cones of Banach algebra morphisms and double mapping cones of pairs of Banach algebra morphisms.

\begin{definition}
\label{def-mapping-cone}
Let   $\varphi\colon A \to B$ be a morphism of Banach algebras.\footnote{Throughout the paper, our morphisms are required to be continuous, but not contractive unless otherwise advertised. In particular, for us, an isomorphism of Banach algebras is not required to be isometric.} 
The \emph{mapping cone} $\MC (\varphi)$ is the Banach algebra
\begin{equation*}
    \MC (\varphi)
=
 \bigl  \{\, (a,f)\in A \oplus C([0,1],B)
 :\varphi(a) = f(0)    \,\, \text{and}\,\,  f(1) = 0  \, \bigr\} .
\end{equation*}
\end{definition}

If we define the \emph{suspension} of the Banach algebra $B$ by 
\[
\Susp (B) =  \{\, f\in C([0,1],B)
 :  f(0)   =   0 =  f(1)   \, \bigr\} ,
 \]
 then there are Banach algebra morphisms 
\[
\Susp (B) \stackrel{\iota} \longrightarrow \MC(\varphi) \stackrel{\pi} \longrightarrow A
\]
given by inclusion and projection, and they induce a six-term mapping cone exact sequence in $K$-theory
\begin{equation}
    \label{eq-mapping-cone-six-term-sequence}
\xymatrix{
K_0(\MC (\varphi)) \ar[r]^-{\pi_*} & K_0(A)  \ar[r]^{\varphi_*} & K_0(B) \ar[d]^{\iota_*} \\
\ar[u]^{\iota_*} K_1(B) &\ar[l]^{\varphi_*} K_1(A)  & \ar[l]^-{\pi_*}  K_1(\MC (\varphi)).
}
\end{equation}
This uses the suspension and periodicity isomorphisms 
\[
K_0 (\Susp (B)) \cong K_1 (B) 
\quad \text{and} \quad 
K_1 (\Susp (B)) \cong K_0 (B) .
\]

\begin{definition}
Given a diagram
of Banach algebras and Banach algebra morphisms  of the form 
\begin{equation}
\label{eq-half-square}
\xymatrix{
     & B \ar[d]^{\varphi }\\
C \ar[r]_{\psi  } \ar[r]& D
}
\end{equation}
we  define the \emph{double mapping cone} $\DMC (\varphi,\psi)$ to be the Banach algebra
\begin{multline*}
    \DMC (\varphi,\psi)
=
 \Bigl  \{\, (b,f,c)\in B \oplus C([0,1],D)\oplus C 
 \\
 :\varphi(b) = f(0)    \,\, \text{and}\,\, \psi(c) = f(1)  \, \Bigr\} .
\end{multline*}
\end{definition}

There is an obvious surjective morphism   from $\DMC (\varphi,\psi)$ to $B\oplus C$, and its kernel is $\Susp(D)$. Associated to the surjection is therefore a $K$-theory $6$-term exact sequence
\begin{equation}
    \label{eq-dmc-six-term-sequence}
\xymatrix@C=20pt{
K_0(\DMC (\varphi,\psi)) \ar[r] & K_0(B) \oplus K_0(C) \ar[r] & K_0(D) \ar[d]\\
\ar[u] K_1(D) &\ar[l] K_1(B) \oplus K_1(C) & \ar[l] K_1(\DMC (\varphi,\psi))
}
\end{equation}
that is analogous to \eqref{eq-mapping-cone-six-term-sequence}.

\begin{definition}
\label{def-m-v-property}
Suppose the diagram \eqref{eq-half-square} is part of a commuting square of Banach algebras and Banach algebra morphisms
\begin{equation}
\label{eq-commuting-square}
\xymatrix{
A \ar[r]^{\beta}\ar[d]_{\gamma} & B \ar[d]^{\varphi }\\
C \ar[r]_{\psi } \ar[r]& D .
}
\end{equation}
 The  associated \emph{canonical morphism}  is  
\begin{gather*}
A \longrightarrow \DMC (\varphi,\psi)
\\
a \longmapsto (\beta(a), f_a, \gamma (a)),
\end{gather*}
where $f_a\colon [0,1]\to B$ is the constant function with value $\varphi (\beta(a))=\psi (\gamma(a))$.
We shall say that   \eqref{eq-commuting-square}
has the \emph{Mayer-Vietoris property} if the canonical morphism  induces an isomorphism in $K$-theory.
\end{definition}

If the diagram \eqref{eq-commuting-square} has the Mayer-Vietoris property, then by substituting  $K_*(A)$ for   $K_*(\DMC(\varphi,\psi))$ in the six-term sequence \eqref{eq-dmc-six-term-sequence} we obtain the $6$-term Mayer-Vietoris exact sequence
\begin{equation}
    \label{eq-basic-m-v-sequence}
\xymatrix{
K_0(A) \ar[r] & K_0(B) \oplus K_0(C) \ar[r] & K_0(D) \ar[d]\\
\ar[u] K_1(D) &\ar[l] K_1(B) \oplus K_1(C) & \ar[l] K_1(A).
}
\end{equation}
This explains the terminology in Definition~\ref{def-m-v-property}.   

We shall need two simple facts about the Mayer-Vietoris property related to a  commuting cube of Banach algebras and Banach algebra morphisms
\begin{equation}
    \label{eq-cube}
\xymatrix@!0@C=27pt{
& A_1 \ar[dl] \ar[rr]\ar'[d][dd]
& & B_1 \ar[dl]\ar[dd] \\
A_2  \ar[rr]\ar[dd]
& & B_2  \ar[dd] \\
& C_1\ar[dl]  \ar'[r][rr] & & D_1 \ar[dl]
\\
C_2 \ar[rr]  & & D_2 
}
\end{equation}

\begin{lemma}
\label{lem-mv-cube-back-to-front}
If either the front face or the back face in \eqref{eq-cube} has the Mayer-Vietoris property, and if all the morphisms from the back face to the front face induce isomorphisms in $K$-theory, then the opposite face has the Mayer-Vietoris property, too. \qed
\end{lemma}

\begin{lemma}
\label{lem-mv-cube-mapping-cone}
If both the front face and the back face in \eqref{eq-cube} have the Mayer-Vietoris property, then the square composed of the mapping cone algebras for the morphisms from the back face to the front face has the  Mayer-Vietoris property, too.\qed
\end{lemma}

\begin{definition} The  diagram \eqref{eq-commuting-square} has the \emph{pullback property}, or is a  \emph{pullback square}, if the morphism of Banach algebras
\begin{equation*}
A \longrightarrow 
\bigl \{ \,
(b,c)\in B\oplus C : \varphi (b) = \psi  (c) \, \bigr \}
\end{equation*}
defined by the formula $a  \mapsto (\beta(a), \gamma(a))$ is an isomorphism.
\end{definition}

Our objective for the remainder this section is to prove the following result:

\begin{theorem}
\label{thm-improved-m-v-thm}
If the commuting  square of Banach algebras  \eqref{eq-commuting-square} is a pullback square,  if 
\[
\varphi [ B] + \psi[C] = D ,
\]
and if one or both of the morphisms $\varphi $ or $\psi $ has dense image, then the square has the Mayer-Vietoris property. 
\end{theorem}

\begin{remark}
It is   simple to prove that the commuting square \eqref{eq-commuting-square} is a pullback square, and if $\varphi[B]{=}D$,  or if $\psi[C]{ =} D$, then  the square has the Mayer-Vietoris property (for instance, the algebraic argument in \cite[\S\S1--3]{Milnor71} is easily adapted to topological $K$-theory).
The stronger result in Theorem~\ref{thm-improved-m-v-thm}  is certainly known to experts in the $K$-theory of Banach algebras, but it is difficult to find a written account.
\end{remark}

The proof of Theorem~\ref{thm-improved-m-v-thm} involves the definition of $K$-theory in terms of homotopy groups, some elementary ideas from homotopy theory,  and an implicit function theorem in the Banach space context.

To begin, recall that the $K$-theory groups of $A$  may be defined by $K_j(A)=\varinjlim \pi_{j-1}(GL_n (A))$ for $j{>}0$, using the standard embeddings on $GL_n(A)$ into $GL_{n+1}(A)$ \cite[\S3]{Milnor71}.\footnote{As usual, when a Banach algebra ${A}$ does not have a multiplicative identity element, we embed ${A}$ as an ideal in any Banach algebra $\widetilde {{A}}$ that does have one, and  then define $GL_n({A})$ to be the kernel of the morphism $GL_n(\widetilde {A})\to GL_n(\widetilde {{A}} /{A})$.}  .

Next, recall that a continuous map $f\colon  E \to B$ between two topological spaces is a (\emph{Serre}) \emph{fibration} if $f$ satisfies the homotopy lifting property \cite[Ch.1, Sec.7]{Whitehead78} for maps from a cube of any finite dimension into $B$.  If we fix a basepoint $e\in E$, set $b= f(e)$, and form the  fiber $F = f^{-1}[b]$, then associated to a Serre fibration there is an exact sequence of homotopy groups
\begin{equation*}
    \cdots \to \pi_j(F,e) \to \pi_j(E,e) \to \pi_j(B,b) \to
    \pi_{j-1} (F,e) \to \pi_{j-1}(E,e)\to  \cdots .
\end{equation*}

If $A\to A/J$ is a surjective morphism of Banach algebras, then the induced morphism 
\[
GL_n({A}) \longrightarrow GL_n({A/J})
\]
is a Serre fibration for every $n$, with fiber $GL_n({J})$. The $K$-theory exact sequence
\begin{equation*}
    \cdots \to K_{j+1}({J}) \to K_{j+1}({A}) \to K_{j+1}({A/J}) \to K_{j}(J) \to K_j (A) \to \cdots 
\end{equation*}
is  the  associated  exact sequence of homotopy groups (or rather the direct limit over $n$ of these). 
The sequence is $2$-periodic in $j$, leading to the six-term sequences mentioned earlier. 

The above leads immediately to the following sufficient condition for a commuting square \eqref{eq-commuting-square} with the pullback property to have the Mayer-Vietoris property (the pullback property allows us to identify  $GL_n(A)$ with the fiber of the map in the lemma):

\begin{lemma}
Suppose given a commuting   square of Banach algebras  \eqref{eq-commuting-square} with the pullback property.  If for every $n$ the continuous map 
\begin{equation}
    \label{eq-fibration-map}
\pi  \colon GL_n(B)\times GL_n(C)\longrightarrow GL_n(D)
\end{equation}
given by the formula 
\[
\pi (S,T) \longmapsto  \varphi (S)^{-1}\psi ( T)
\]
 is a Serre fibration, then the square has the Mayer-Vietoris property. \qed
\end{lemma}

Thanks to this, to prove Theorem~\ref{thm-improved-m-v-thm} it suffices to prove the following result: 

\begin{theorem}
\label{thm-improved-m-v-thm2}
Under  the hypotheses of Theorem \textup{\ref{thm-improved-m-v-thm}},  
 the continuous map  $\pi $ in \eqref{eq-fibration-map}
 is a Serre fibration.
\end{theorem}

 The main step in the proof of Theorem~\ref{thm-improved-m-v-thm2} will be as follows: 

\begin{lemma}\label{lem-banfib1}
Under  the hypotheses of Theorem \textup{\ref{thm-improved-m-v-thm}},  
 the continuous map  $\pi $ in \eqref{eq-fibration-map}
has  a local right inverse, defined in a neighborhood $U$ of the identity in $GL_n(D)$.  That is, there is a continuous map 
\[
\sigma\colon  U \longrightarrow  GL_n (B)\times GL_n(C) 
\]
for which the composition $\pi \circ \sigma$  is the inclusion of $U$ into $GL_n(D)$.
\end{lemma}

\begin{proof}[Proof of Theorem~\ref{thm-improved-m-v-thm2}, assuming Lemma~\ref{lem-banfib1}]
The group 
\[G = GL_n(B){\times} GL_n(C)
\]
acts on itself by right multiplication,   and on  $GL_n(D)$ by the formula
\[
R\cdot (S,T) = \varphi (S)^{-1} R \psi (T) .
\]
The map $\pi $  in \eqref{eq-fibration-map} is $G$-equivariant, so its image is   a union of $G$-orbits.  
This, together with the assumption that one or both of $\varphi $ or $\psi $ has dense image   and a computation using the exponential map, shows that each $G$-orbit in $GL_n(D)$ is dense in each component of $GL_n(D)$ that it intersects. 

The existence of the local section $\sigma$ implies that the image of $\pi $ is open. Now, the complement of the image is also a union of $G$-orbits,  and so from the above it follows that the image is a union of components in $GL_n(D)$.

The fiber of $\pi $ over the identity element of $GL_n(D)$ is 
\[
GL_n(A)\cong \{ \,  (\beta (Q), \gamma(Q))\in GL_n(B){\times} GL_n(C) : Q \in GL_n (A)\,\}.
\]
If $U\subseteq GL_n(D)$ is the open set on which the section $\sigma$ is defined, and if  for $(S,T)\in \pi^{-1} [ U]$ we set 
\[
\theta(S,T) = \sigma (\pi (S,T))\cdot (S^{-1}, T^{-1}),
\]
using multiplication in  $GL_n(B){\times} GL_n(C)$, then in fact $\theta(S,T)\in GL_n(A)$ and 
the formula 
\[
(S,T) \longmapsto \bigl ( \theta(S,T) , \pi (S,T)\bigr )
\]
defines a   $GL_n(A)$-equivariant homeomorphism from $\pi ^{-1}[U]$ to $GL_n(A)\times U$ (for the obvious left $GL_n(A)$-actions).   So $\pi ^{-1}[U]$ is a trivial   principal $GL_n(A)$-bundle over $U$.  Using right $G$-equivariance, we find that the map 
\[
\pi \colon  GL_n(B){\times}GL_n(C) \longrightarrow GL_n(D)
\]
is a principal  $GL_n(A)$-bundle over its image,  an open and closed subset of $GL_n(D)$. Since principal bundles are fibrations  the proof is complete.
\end{proof}

We shall construct the local right inverse in Lemma~\ref{lem-banfib1} using a version of the implicit function theorem, the essential part of which is the following simple nonlinear version of the open mapping theorem  (we have simplified the original statement   by specializing it).

\begin{thm}[{\cite[Thm~1]{Graves50}}]
\label{thm-gravesthm1}
Let $\mathcal{X}$ and $\mathcal{Y}$ be Banach spaces and  let $F$ be a continuous function from a neighborhood of $0\in \mathcal{X}$ into $\mathcal{Y}$ such that $F(0)=0$.   If there exists a surjective continuous linear operator $L\colon \mathcal{X} \to \mathcal{Y}$  such that 
\[
\bigl \|F(x_1) - F(x_2) - L(x_1 {-} x_2) \bigr \| =  \mathcal{O} \bigl ( (\|x_1\|{+}\|x_2\|)\|x_1{-}x_2\|\bigr ) ,
\]
then the element $0\in \mathcal{Y}$ lies in the interior of the image of $F$.
\end{thm}

In order to construct a local section from this result we use the well-known Bartle-Graves theorem: 

\begin{thm}[{See  \cite[Thm~4]{BartleGraves52}}]\label{thm-bartle-graves}
Every continuous and surjective linear operator between Banach spaces has a continuous \textup{(}but not necessarily linear\textup{)} right inverse, which may be chosen so that it maps bounded subsets of $\mathcal{Y}$ to bounded subsets of $\mathcal{X}$. 
\end{thm}

\begin{corollary}
\label{cor-cintinuous-implicit-function-thm}
Let $\mathcal{X}$ and $\mathcal{Y}$ be Banach spaces and  let $F$ be a continuous function from a neighborhood of $0\in \mathcal{X}$ into $\mathcal{Y}$ such that $F(0)=0$.   If there exists a surjective continuous linear operator $L\colon \mathcal{X} \to \mathcal{Y}$  such that 
\[
\bigl \|F(x_1) - F(x_2) - L(x_1 {-} x_2) \bigr \| =  \mathcal{O} \bigl ( (\|x_1\|{+}\|x_2\|) \|x_1{-}x_2\|\bigr )
\]
then $F$ has a continuous local right-inverse,  defined on some neighborhood of $0\in \mathcal{Y}$.
\end{corollary}

\begin{proof}  
Denote by $C_b (\mathcal{Y},\mathcal{X}) $ the Banach space of bounded, continuous functions from $\mathcal{Y}$ to $\mathcal{X}$ (the norm is the supremum norm associated to the norm on $\mathcal{X}$) and define  $C_b (\mathcal{Y},\mathcal{Y}) $ similarly.
Apply Theorem~\ref{thm-gravesthm1} to the continuous map 
\[
\boldsymbol{F}\colon 
C_b (\mathcal{Y},\mathcal{X}) 
\longrightarrow 
C_b (\mathcal{Y},\mathcal{Y})
\]
given by composition with $F$ (it is defined on the open subset of $C_b (\mathcal{Y},\mathcal{X}) $ consisting of functions from $\mathcal{Y}$ to the open subset of $\mathcal{X}$ on which $F$ is defined) and   the continuous linear operator 
\[
\boldsymbol{L}: C_b(\mathcal{Y},\mathcal{X}) \to C_b(\mathcal{Y},\mathcal{Y})
\]
given by composition with $L$ (it follows from the Bartle-Graves theorem that $\boldsymbol{L}$ is  surjective).  
\end{proof}

\begin{proof}[Proof of Lemma~\ref{lem-banfib1}]
Define a continuous map from a neighborhood of zero in the Banach space  $M_n(B)\oplus M_n(C)$ into the Banach space $M_n(D)$ by the formula 
\[
F(S,T) = \log \bigl (\exp(-\varphi (S))\exp(\psi (T))\bigr ) .
\]
Define a bounded linear map from $M_n(B)\oplus M_n(C)$ into $M_n(D)$ by 
\[
L(S,T) = \psi  (T) - \varphi (S) .
\]
It is surjective, thanks to the hypotheses of Theorem~\ref{thm-improved-m-v-thm}.
Now apply Corollary~\ref{cor-cintinuous-implicit-function-thm}. Compose the local right inverse that  it provides with  the logarithm and  exponential maps on $GL_n(D)$ and $M_n(B)\oplus M_n(C)$ respectively, to obtain a local right inverse to the map \eqref{eq-fibration-map}.
\end{proof}

\section{The Gelfand Theorem and Continuity}
\label{sec-gelfand-theorem}

A famous theorem of Gelfand asserts that an element of a commutative Banach algebra with unit is invertible if and only if its Gelfand transform is invertible \cite[Ch I, \S4]{GelfandRaikovShilov64}.  This has the following consequence for $K$-theory, which we shall use several times:

\begin{theorem}
\label{thm-gelfand-karoubi}
If a morphism $\alpha\colon A {\to} B$ of commutative unital Banach algebras has dense image and induces a homeomorphism 
\[
\alpha^* \colon \Spec (B)\stackrel \cong \longrightarrow  \Spec(A),
\]
then it induces an isomorphism on $K$-theory.
\end{theorem}

Of course this is an immediate consequence of Novodvorskii's theorem, but the dense image assumption allows for a simple, direct proof, based on the following well-known result: 

\begin{thm}[Karoubi Density  {\cite[Exercise II.6.15]{Karoubi78}}]
Let $\alpha\colon A\to B$ be an injective  morphism of Banach algebras whose image is dense in $B$. If for every $n$ 
\[
\alpha[GL_n(A)] = M_n (\alpha[A]) \cap GL_n (B) ,
\]
then    the group homomorphism 
$
\alpha \colon GL_n(A)\to GL_n(B)
$
induces isomorphisms $\pi_j(GL_n(A)){\to} \pi_j(GL_n(B))$ for every $j$ and every $n$, and so in particular $\alpha$ induces an isomorphism in $K$-theory.\qed
\end{thm}

\begin{proof}[Proof of Theorem~\ref{thm-gelfand-karoubi}]
By adjoining units, if necessary, we may assume that $A$ and $B$ are unital. 
The hypothesis in the theorem implies that every multiplicative linear functional on $A$ vanishes on $\operatorname{kernel}(\alpha)$, since it factors (uniquely) through $\alpha$.  It follows that the induced morphism 
\[
\bar \alpha\colon A / \operatorname{kernel} (\alpha) \longrightarrow B
\]
also satisfies the hypothesis of the theorem.

If $f\in A / \operatorname{kernel} (\alpha) $ and  if the element  $\bar \alpha(f) \in B$ is invertible, then by this hypothesis and by Gelfand's theorem,  $f$   is   invertible in  $A / \operatorname{kernel} (\alpha) $.  
Using determinants, we find that if $[f_{ij}]$ is an $n{\times}n$ matrix over $A/\operatorname{kernel} (\alpha)$, and if  for which $[\bar \alpha (f_{ij})]$ is an invertible matrix over $B$, then $[f_{ij}]$ is itself invertible.  So the Karoubi density theorem applies to $\bar\alpha$.

Now if $a\in \operatorname{kernel} (\alpha)$, then Gelfand's theorem implies that   $1{+}a$ is invertible in $A$. Using 
\[
( 1 + a) ( 1 + a') = 1 + a + a' +aa'
\]
we find that the inverse has the form $1{+}a'$ with $a'\in \operatorname{kernel}(\alpha)$.
Using determinants, again, we find that 
\[
GL_n(\operatorname{kernel}(\alpha)) = I + M_n ( \operatorname{kernel}(\alpha)).
\]
So $GL_n(\operatorname{kernel}(\alpha)) $ is contractible for all $n$, and hence
\[
K_*(\operatorname{kernel}(\alpha)) = 0.
\]
The theorem follows from this, the six-term exact sequence in $K$-theory, and the fact that $\bar \alpha$ induces an isomorphism in $K$-theory.
\end{proof}

Next, suppose given a diagram of Banach algebras and Banach algebra morphisms
\[
A_1 \stackrel {\alpha_1} \longrightarrow A_2 \stackrel{\alpha_2} \longrightarrow A_3 \stackrel {\alpha_3} \longrightarrow \cdots 
\]
in which the   the norm of each morphism $\alpha_r$ is $1$, or less.  The algebraic vector space direct limit is in fact an associative algebra, and the formula 
\[
\| a \| = \lim_{s\to \infty} \| \alpha_{s+r,r}(a)\|_{A_{s+r}} ,
\]
where $a\in A_r$ and where $\alpha_{s+r,r}$ is the composition of the morphisms in the diagram from $A_r$ to $A_{s+r}$, defines a submultiplicative seminorm.  Forming the quotient by the ideal of elements with seminorm $0$, and then completing, we obtain the \emph{Banach algebra direct limit}, which we shall write as $\varinjlim A_r$.  

There are canonical morphisms from each $A_r$ into the direct limit. We shall make frequent use of the following well-known continuity property of $K$-theory, which may be proved using a  variation of the Karoubi density theorem.

\begin{theorem}
\label{thm-k-theory-continuity}
The  induced morphism
\begin{equation*}
\varinjlim K_*(A_r ) \longrightarrow K_* ( \varinjlim A_r)
\end{equation*}
is an isomorphism of abelian groups. \qed
\end{theorem}

A modest first application is the reduction of Novodvorskii's theorem to the case of (topologically) finitely generated Banach algebras.

\begin{corollary}
\label{cor-reduction-to-f-g}
If the Gelfand transform induces an isomorphism in $K$-theory for every unital and finitely generated commutative Banach algebra, then it induces an   isomorphism in $K$-theory for every separable \textup{(}in the  sense of general topology\textup{)} unital commutative Banach algebra.
\end{corollary}

\begin{proof}
A general unital commutative Banach algebra is the direct limit of its finitely generated subalgebras, and the Gelfand spectrum is the inverse limit of the Gelfand spectra of those subalgebras. Since the canonical morphism
\[
 \varinjlim C(X_\alpha) \longrightarrow  C(\varprojlim X_\alpha) 
\]
is an isomorphism of Banach algebras, the continuity property of $K$-theory  completes the proof.
\end{proof}

The obvious modification of  the  results of this section to handle general directed systems reduces  Novodvorskii's theorem for not-necessarily-separable Banach algebras  to the finitely generated case, too.

Finally, we shall take advantage of the compatibility of the   Mayer-Vietoris property with direct limits.

\begin{lemma}
\label{lem-limits-and-m-v}
Suppose given a sequence of commuting cubes 
\[
\xymatrix@!0@C=30pt{
& A_n \ar[dl] \ar[rr]\ar'[d][dd]
& & B_n \ar[dl]\ar[dd] \\
A_{n+1}  \ar[rr]\ar[dd]
& & B_{n+1}  \ar[dd] \\
& C_n\ar[dl]  \ar'[r][rr] & & D_n\ar[dl]
\\
C_{n+1} \ar[rr]  & & D_{n+1} 
}
\]
with the morphisms from the back to front faces of norm $1$ or less, giving rise to a commuting square
\[
\xymatrix@C=30pt{
\varinjlim A_n \ar[r] \ar[d]  & \varinjlim B_n \ar[d] \\
\varinjlim C_n \ar[r]  \ar[r]& \varinjlim D_n
}
\]
If all the back faces of the cubes have the Mayer-Vietoris property, then so does the direct limit square. 
\qed
\end{lemma}

\section{The Banach Algebra of a Polynomially Convex Compact Set}

\begin{definition}
A compact subset $X $ of $\C^n$ is \emph{polynomially convex} if it includes all points $w\in \C^n$ with the property that
\[ | p(w)| \le \max \{ \, |p(z)| : z \in X\,\} 
\]
for every complex (holomorphic) polynomial  function  $p $ on $\C^n$.
\end{definition}

See for example Hormander's monograph  \cite{Hormander90}.  Every compact convex set is polynomially convex. The polynomially convex compact subsets of $\C$ are   the compact subsets with connected complements, but in higher dimensions there is a much richer set of possibilities.  For instance every compact subset of $\R^n$ is polynomially convex in $\C^n$.

\begin{definition}
Let $X$ be a polynomially convex compact subset of $\C^n$. We shall denote by $B(X)$ the closure in the supremum norm on $X$ of the algebra of complex polynomial functions on $X$.
\end{definition}

Obviously $B(X)$ is a commutative Banach algebra under pointwise multiplication.  Each point   $z\in X$  determines a multiplicative linear functional   $\varepsilon _z \colon B(X)\to \C$ by evaluation at $z$, and it is easily checked that these are the \emph{only} multiplicative linear functionals. Hence:

\begin{lemma}
Let $X$ be a polynomially convex compact subset of $\C^n$.  The map $z\mapsto \varepsilon_z$
is a homeomorphism from $X$ onto the   Gelfand spectrum of $B(X)$.    \qed 
\end{lemma}
 
The Gelfand transform may therefore be identified  with the inclusion of $B(X)$ into the algebra $C(X)$ of all continuous complex-valued functions on $X$, and hence the Novodvorskii theorem for $B(X)$ states the following:
 
\begin{theorem}
\label{thm-novodvorskii-for-b-of-x}
Let $X$ be a polynomially convex compact subset of $\C^n$. The inclusion of $B(X)$ into $C(X)$ induces an isomorphism in $K$-theory.
\end{theorem}

We shall prove this result in the next two sections.  Actually the full Novodvorskii theorem follows easily from Theorem~\ref{thm-novodvorskii-for-b-of-x}: 

\begin{theorem}
\label{thm-reduction-to-b-of-x}
If  the Gelfand transform for $B(X)$ induces an isomorphism in $K$-theory for every compact and polynomially compact subset  $X\subseteq \C^n$ and every $n$, then the Gelfand transform induces an isomorphism in $K$-theory for every commutative unital Banach algebra.
\end{theorem}

\begin{proof}
If $A$ is  finitely generated by $a_1,\dots, a_n$, so that the polynomials in these elements are dense in $A$, then the morphism $\varphi\mapsto (\varphi(a_1),\dots, \varphi(a_n))$ is a homeomorphism from the Gelfand spectrum of $A$ onto a polynomially convex  compact subset $X$ of $\C^n$.  

The Gelfand transform maps $A$ to a dense subalgebra of $B(X)$, and it follows from Theorem~\ref{thm-gelfand-karoubi} that the induced map on $K$-theory is an isomorphism.  So assuming that the Gelfand transform induces an isomorphism in $K$-theory for $B(X)$, it does so for $A$ as well.  The result now follows from Theorem~\ref{cor-reduction-to-f-g}.
\end{proof}

\section{The Mayer-Vietoris Property for Polynomially Convex Compact Sets} 

The purpose of this section is to prove the following result:

\begin{theorem}
\label{thm-m-v-for-polynomially-convex-sets}
Let $X\subseteq \C^n$ be a polynomially convex compact set. Let $\alpha$ be an $\R$-linear functional on $\C^n$ and   let $c \in \R$. If 
\[
X' = \{ \, z \in X : \alpha(z) \le c \, \} 
\quad \text{and} \quad 
X'' = \{ \, z\in X : \alpha(z) \ge c \, \},
\]
then the commuting square  \begin{equation*}
\xymatrix{
B(X) \ar[r] \ar[d]  & B(X') \ar[d] \\
B(X'') \ar[r]  \ar[r]& B(X' \cap X'')
}
\end{equation*} 
of restriction morphisms has the Mayer-Vietoris property.
\end{theorem}

We shall prove the theorem through a sequence of reductions and approximations, most of them involving the following concept:

\begin{definition}
\label{def-compact-polynomial-polyhedron}
A compact set $X\subseteq \C^n$ is a \emph{polynomial polyhedron} if there are complex polynomial functions $p_1,
\dots, p_k$ on $\C^n$ such that 
\[
X=\{ \, z\in \C^n : |p_1(z)|,  \dots, | p_k(z)| \le 1 \,\}.
\]
\end{definition}

\begin{lem}[See   for example {\cite[Lemma~2.7.4]{Hormander90}}]
\label{lem-localpolyhedra1}
Let $X \subseteq \bC^n$ be a polynomially convex compact set and  let $U\subseteq \C^n$ be an open set containing $X$. There is a compact polynomial polyhedron $L $ such that $X \subseteq L \subseteq U$. \qed
\end{lem}

\begin{corollary} Every polynomially convex compact set is the intersection of a nested sequence of compact polynomial polyhedra. \qed
\end{corollary}

It follows from the corollary and from Lemma~\ref{lem-limits-and-m-v} that the special case of Theorem~\ref{thm-m-v-for-polynomially-convex-sets} in which $X$ is a compact polynomial polyhedron  implies the general case. So \emph{from now on we shall assume that $X$ is a compact polynomial polyhedron, as in Definition~\ref{def-compact-polynomial-polyhedron}.}

Choose a positive number $R$   so that 
\[
(z_1,\dots, z_n) \in X \quad \Rightarrow \quad  |z_1|, \dots , |z_n| \le R
\]
and form the compact polydisk
\[
W = \{\, (z,w)\in \C^{n+k} :  |z_1|, \dots , |z_n| \le R\,\,\,\text{and} \,\,\, |w_1|,\dots, |w_k |\le 1\,\} .
\]
Let 
\[
 Z = \{\, (z,w)\in W :  p_1(z) = w_1   , \dots, p_k(z) = w_k\, \}  
\]
and define a polynomial map 
$\mu\colon \C^n \to \C^{n+k}$ by
\begin{equation}
\label{eq-oka-map}
\mu(z) = (z, p_1(z),\dots, p_k(z) ) .
\end{equation}
Observe that  $\mu$ restricts to a homeomorphism from $ X$ to $Z$. The inverse of this homeomorphism  is given by the coordinate projection from $\C^{n+k}$ to $\C^n$. Since $\mu$ and this inverse are both polynomial maps, it is evident that composition with $\mu$ gives an isomorphism of Banach algebras
\[
\mu^* \colon  B(Z)\stackrel\cong \longrightarrow B(X).
\]
As we shall see, the advantage of doing so is that $Z$  is defined, as a subset of the polydisk $W$, by  polynomial equations, namely the equations  $q_j (z,w) =0$, where 
\begin{equation}
    \label{eq-q-polynomials}
q_j (z, w ) = p_j(z) - w_j \qquad (j=1,\dots ,k).
\end{equation}

We shall need  sets $W_m$, $X_m$ and $Z_m$ ($m=1,2,\dots$), that approximate   $W$, $X$ and $Z$  from the outside. For $m \ge  1$   we define
\begin{multline*}
W_{m} =  \{ \,  (z,w)\in \C^{n+k} : 
 |z_1|, \dots , |z_n| \le R{+}{1}/{m} \\
 \,\,\text{and} \,\, |w_1|,\dots, |w_k |\le 1{+}{1}/{m}
 \,  \} 
\end{multline*}
and then define 
\[
Z_{m} = \{\, (z,w) \in W_{m} : q_1(z,w) = \cdots = q_k(z,w) =0 \,\}
\]
and 
\[
X_{m} = \{ \, z\in \C^n : \mu(z) \in Z_{m} \,\} .
\]
Each $X_{m+1}$ lies in the interior of $X_{m}$, and $X$ is the intersection of all $X_{m}$. Similarly for $W_{m}$ and $Z_{m}$.
We   also define 
\[
X_{m}' = \{ z\in X_m : \alpha(z) \le  c{+}{1}/{m}\,\}
\]
and
\[
X_{m}'' = \{ z\in X_m : \alpha(z) \ge  c{-} {1}/{m}\,\} ,
\]
and similarly for $W_{m}'$, $W_{m}''$,  $Z_{m}'$ and $Z_{m}''$, using the same bounds on $\alpha(z)$.  Once more we obtain nested decreasing sequences of compact sets, with each set contained in the interior of its predecessor.

The following result is   relevant to the Mayer-Vietoris property in view of Theorem~\ref{thm-improved-m-v-thm}.

\begin{proposition}\label{prop-uniform-theorem-b-polydisc}
For every $m$,   
\begin{multline*}
\operatorname{Image} \bigl [B(W_{m}'){\to} B(W_{m}'{\cap} W_{m}'')\bigr ] 
\\ +  \operatorname{Image}\bigl [  B(W_{m}''){\to} B(W_{m}'{\cap }W_{m}'')\bigr ] =  B(W_{m}' {\cap} W_{m}''),
\end{multline*}
where the morphisms are restrictions.
\end{proposition}

In the proof we shall use the fact that    $B(W_{m}'{\cap }W_{m}'')$ includes a dense family of functions that extend   holomorphically to a neighborhood of $W_{m}'{\cap }W_{m}''$   (namely the polynomial functions, for instance), and that every function that is holomorphic in a neighborhood of $W_{m}'$ or $W_{m}'' $ restricts to a function in $B(W_{m}')$ or $B(W_{m}'')$, respectively.  The latter is a general fact about polynomially compact convex sets (the Oka-Weil theorem) but it may be proved directly in the present simple case.

\begin{proof}
It suffices to prove the following approximation result. There is a constant $C>0$ such that if $f$ is holomorphic in a neighborhood of $W_m'{\cap}W_m''$,   then there are functions $f' \in B(W_m')$ and $f''\in B(W_m'')$ such that 
\begin{equation}
    \label{eq-cocycle-relation}
  f \vert_{W_m' \cap W_m''} =
f'\vert_{W_m' \cap W_m''} + f'' \vert_{W_m' \cap W_m''}
\end{equation}
and 
\begin{equation}
\label{eq-norm-estimate-for-f-and-g}
\|f'\|_{W_m'}, \|f''\|_{W_m''} \leq C \|f\|_{W_m' \cap W_m''} .
\end{equation}

Let $\sigma: \R \to \R $ be a smooth function such that $\varphi\equiv  0$ in a neighborhood of $(-\infty , c{-}1/m]$  and $\varphi\equiv 1$  in a neighborhood of $[c{+}1/m,\infty)$.  Then define $\varphi\colon \C^{n+k}\to \R $ by $\varphi (z) = \sigma(\alpha(z))$.

Choose linear coordinates $v_1,\dots, v_{n+k}$ on $\C^{n+k}$ with $\alpha(z) = \Re(v_1)$, and given $f$ as   above, let  $V$ be a compact polydisk that includes $W_m$ within its interior, that  is small enough that $f$ is defined in a neighborhood of $V'{\cap}V'' = V \cap \{ \,c{-}1/m {\le } \alpha(z) {\le } c{+}1/m \,\} $, and that is small enough that 
\begin{equation}
    \label{eq-estimate-on-v-from-w}
\Bigl \|\frac {\partial \varphi}{\partial \bar v_1} \cdot f \Bigr \|_{V} \leq 2\Bigl \| \frac {\partial \varphi}{\partial \bar v_1}\cdot  f\Bigr \|_{W_m}  .
\end{equation}
Note that even though $f$ is only defined near $V'{\cap }V''$, its product with $\partial \varphi/\partial \bar v_1$ extends by zero to a smooth function in a neighborhood of $V$. Next, let $D$ be the projection of $V$ onto the first complex coordinate (it is a compact disk in $\C$).  Then define  
\[ 
g (v_1,\dots ,v_{n+k})= \frac{1}{\pi }\int_{D}  \Bigl ( \frac{\partial \varphi}{\partial \bar v_1}\cdot f \Bigr ) (v,v_2,\dots,v_{n+k})   (v - v_1)^{-1} d\lambda(v),
\]
where the integration is with respect to Lebesgue measure on the disk.
The function $g$ is defined and smooth on the interior of $V$, it is holomorphic there in the variables $v_2,\dots , v_{n+k}$, and 
\[
 \frac{\partial g}{\partial \bar v_1}
 =
 \frac{\partial \varphi}{\partial \bar v_1}\cdot f .
 \]
 (see for example \cite[Theorem 1.2.2]{Hormander90}). Moreover since $(v-v_1)^{-1}$ is uniformly integrable as a function of $v\in D$ as $v_1$ ranges over $D$, 
\begin{equation*}
\| g\|_{W_m}    
\le \text{constant}_1\cdot 
    \Bigl \|\frac {\partial \varphi}{\partial \bar v_1} \cdot f \Bigr \|_{V}
    \le \text{constant}_2 \cdot 
     \Bigl \|\frac {\partial \varphi}{\partial \bar v_1} \Bigr \|_{V}\cdot \| f\|_{W_m'\cap W_m''},
\end{equation*}
 where the second inequality uses \eqref{eq-estimate-on-v-from-w} and both constants are independent of $f$.
 
Now set 
\[
f' = \varphi f  - g
\quad \text{and} \quad 
f'' = (1 {-}\varphi )f + g
\]
These are holomorphic in   neighborhoods of $W_m'$ and $W_m''$ respectively (here the functions $\varphi f$ and $(1{-}\varphi) f$ are extended by zero to $W_m'$ and $W_m''$, respectively) and they satisfy \eqref{eq-cocycle-relation} and \eqref{eq-norm-estimate-for-f-and-g}, so the proof is complete.
 \end{proof}

Unfortunately it is not easy to prove the counterpart of Proposition~\ref{prop-uniform-theorem-b-polydisc} for the spaces $X_m$ or $Z_m$.  Instead we shall replace the Banach algebras $B(X_m)$, etc, with others that are easier to handle.

\begin{definition}
Let  $V$ be a polynomially convex compact subset of $\C^{n+k}$, and let $Y$ be the subset of $V$ on which all the polynomials $q_1,\dots q_k$ vanish. We shall write 
\[
I(V,Y) = \{\, f\in B(V) : f \vert _Y = 0\,\}
\]
and  
\[
A(V,Y) 
=  B(V) / I(V,Y).
\]

\end{definition}

The ideal $I(V,Y)$ in $B(V)$  is the kernel of the obvious restriction morphism 
$
\res \colon B(V) {\to} B(Y)
$.

\begin{lemma}
\label{lem-gelfand-karoubi-a-versus-b1}
The associated    morphism of Banach algebras 
\begin{equation*}
    \label{eq-restriction-map-from-w-to-z}
\overline \res \colon A(V,Y) \longrightarrow B(Y)
\end{equation*} 
induces an isomorphism in $K$-theory.
\end{lemma}

\begin{proof} 
The Gelfand spectrum of $A(V,Y)$ is mapped to the Gelfand spectrum of $B(V)$  by composition with the quotient morphism 
\[
B(V) \longrightarrow A(V,Y). 
\]
This map of spectra is injective and continuous, and so it is a homeomorphism  onto its image.  

Now the image of the spectrum of $A(V,Y)$ in the spectrum of $B(V)$ is precisely the set of multiplicative linear functionals on $B(V)$ that vanish on $I(V,Y)$.
The evaluation functionals $\varepsilon_y\colon B(V){\to} \C$ associated to points $y\in Y$ obviously vanish on $I(V,Y)$.  On the other hand, the evaluations   at points $v \in V\setminus  Y$  \emph{do not} vanish on $I(V,Y)$, since for every such,  at least one of the polynomials $q_j$  vanishes on $Y $
 but not on $v$. It follows that the image in $V$ of the Gelfand spectrum of $A(V,Y)$ is $Y$. 
 
 The morphism $\overline \res \colon A(V,Y)\to B(Y)$ is  therefore spectrum-pre\-serv\-ing. 
 Since  it also has dense range (consider polynomials),   Theorem~\ref{thm-gelfand-karoubi} applies. 
\end{proof}

This leads us to analyze the commuting the square
\begin{equation}
\label{eq-a-z-commuting-square}
\xymatrix{
A(W_m,Z_m) \ar[r] \ar[d]  & A(W_m',Z_m') \ar[d] \\
A(W_m'',Z_m'') \ar[r]  \ar[r]& A(W_m'{\cap}W_m'',Z_m' \cap Z_m'') ,
}
\end{equation}
Since the notation used above is a bit cumbersome, we shall simplify it and write the square as 
\begin{equation}
\label{eq-a-z-commuting-square2}
\xymatrix{
A(Z_m) \ar[r] \ar[d]  & A(Z_m') \ar[d] \\
A(Z_m'') \ar[r]  \ar[r]& A(Z_m' \cap Z_m'') ,
}
\end{equation}
 which should not cause any confusion.

 \begin{lemma}
 \label{lem-surjectivity-for-a-algebras}
 \begin{multline*}
\operatorname{Image} \bigl [A(Z_m'){\to} A(Z_m'{\cap} Z_m'')\bigr ] 
\\ +  \operatorname{Image}\bigl [  A(Z_m''){\to} A(Z_m'{\cap }Z_m'')\bigr ] =  A(Z_m' {\cap} Z_m''),
\end{multline*}
 \end{lemma}
 
 \begin{proof}
 This is an immediate consequence of Proposition~\ref{prop-uniform-theorem-b-polydisc} since the  $A$-algebras are quotients of their $B$-algebra counterparts.
 \end{proof}

Now form the Banach alegbra 
\[
A(Z_m', Z_m'') = \bigl \{ \, 
(f',f'')\in A(Z_m') \times A(Z_m'') : 
f'\vert_{Z_m'\cap Z_m''} = f''_{Z_m'\cap Z_m''} 
\,\bigr \}.
\]

\begin{proposition}
\label{prop-application-of-inverse-function-m-v}
The commuting square 
\begin{equation*}
\xymatrix{
A(Z_{m}', Z_{m}'') \ar[r] \ar[d]  & A(Z_{m}') \ar[d] \\
A(Z_{m}'') \ar[r]  \ar[r]& A(Z_{m}' \cap Z_{m}'')
}
\end{equation*}  has the Mayer-Vietoris property.  
\end{proposition}

\begin{proof} By construction, 
this 
is  a pullback square.  So the result   follows from 
Lemma~\ref{lem-surjectivity-for-a-algebras} and Theorem~\ref{thm-improved-m-v-thm}.
\end{proof}

To move from this to a proof of Theorem~\ref{thm-m-v-for-polynomially-convex-sets}, we shall need to address the difference between $A(Z_m',Z_m'')$ and $A(Z_m)$, and for this purpose we shall use the following fact. 
The proof will be evident to those familiar with coherent sheaves; we shall give  an alternative   proof from   \cite[Lemma 2]{Allan69}  in an appendix.

\begin{lemma}
\label{lem-due-to-allan}
Let $Y$ be the interior of a compact polynomial polyhedron in $\C^n$ and let 
\[
V = \bigl \{ \, (z,w)\in Y\times\C^k : |w_1|,\dots, |w_k| \le 1\, \bigl \}
\]
Let $p_1,\dots, p_k$ be polynomial functions on $\C^n$, and define polynomial functions $q_1,\dots, q_k$ on $\C^{n+k}$ by $q_j(z,w)= p_j (z) - w_j$. If $f$ is a holomorphic  function defined in a neighborhood of $V$, and if $f $ vanishes on the common zero set of $q_1,\dots, q_k$ in that neighborhood, then there are holomorphic  functions $h_1, \ldots, h_k$ defined near  $V$ such that
\[\pushQED{\qed} 
f =  q_ 1 h_1 + \cdots + q_k h_k.\qedhere
\popQED
\]
\end{lemma}

There is a natural  morphism
from $ A(Z_m)$ to $A(Z_m', Z_m'')$, induced from   restriction  of functions on $W_m$ to $W_m'$ and $W_m''$. The following proposition supplies a near-inverse. 

\begin{lemma}
\label{lem-inverse-to-restriction-map}
For each $m$  there is a Banach algebra morphism 
\[
\varphi_m: A(Z_m', Z_m'') \longrightarrow A(Z_{m+1})
\]
for which the  diagram  
\begin{equation*}
\xymatrix{
A(Z_{m}', Z_{m}'')  \ar[d]\ar[dr]^{\varphi_m}   & A(Z_{m})\ar[d]\ar[l] \\
A(Z_{m+1}', Z_{m+1}'')  & A(Z_{m+1}) \ar[l]
}
\end{equation*}
is commutative \textup{(}the unlabeled morphisms are all induced from restriction of functions\textup{)}.
\end{lemma}

\begin{proof}
Let $([f'], [f''])$ be an element of $ A(Z_m', Z_m'')$, where the square brackets denote equivalence classes of functions on $W_m'$ and $W_m''$.  It follows from the definitions that  the difference $f'{ -} f''$, which is defined on $W_m'{\cap}W_m''$, vanishes on the mutual zero-set  of the polynomials $q_1,\dots, q_k$. 

Let us apply Lemma~\ref{lem-due-to-allan}. By fitting a suitable compact polynomial polyhedron into the interior of  $W_{m}''{ \cap} W_{m}''$ whose interior, in turn,  includes $W_{m+1}''{ \cap} W_{m+1}''$, we find that   there are holomorphic functions $h_1,\dots, h_k$ defined on a neighborhood of  $W_{m+1}''{ \cap} W_{m+1}''$ such that
\[
f' - f'' =  q_1 h_1+\cdots + q_k h_k
\]
on this neighborhood.  In particular, the identity holds on $W_{m+1}''{ \cap} W_{m+1}''$.

By applying Proposition \ref{prop-uniform-theorem-b-polydisc}, to a polydisk that is slightly larger than $W_{m+1}$, we find that for each $j$ there are functions $h_j'\in B(W_{m+1}')$ and $h_j''\in B(W_{m+1}'')$, holomorphic in  neighborhoods of $W_{m+1}'$  and  $W_{m+1}''$, respectively, such that 
\[
h_j = h_j' - h_j''
\]
in a neighborhood of $ W_{m+1}'{ \cap} W_{m+1}''$.

Now define functions $g'\in B(W_{m+1}')$ and $g''\in B(W_{m+1}'')$ by
\[
\textstyle 
 g' =  f' -(  q_1 h_1'+\cdots +q_kh_k' )
 \quad \text{and} \quad 
 g'' = f'' - (q_1h_1''+\cdots + q_k h_k'') .
\]
These  are actually defined and holomorphic in neighborhoods of $W_{m+1}'$ and $W_{m+1}''$, respectively, and they agree on a neighborhood of the intersection.  So they determine a function $g$ that is holomorphic in a neighborhood of $W_{m+1}$, and hence a function  $g$ in $B(W_{m+1})$. 

The associated element $[g]\in  A(Z_{m+1})$ is characterized by the fact that   $g$ is equal to $f'$ on    $Z_{m+1}'$, and  is equal to the  $f''$ on $Z_{m+1}''$.   This is because the morphism from $A(Z_{m+1})$ into $B(Z_{m+1})$ is injective. So $[g]$
depends only on $([f'], [f''])$.  

To complete the proof, the formula
\[
\varphi_m \colon ([f'], [f''])
\longmapsto [g] 
\]
defines  an algebra homomorphism from $A(Z_{m}',Z_m'')$ to $A(Z_{m+1})$ that fits into the commuting diagram in the statement of the proposition.    It follows from the closed graph theorem that $\varphi_m$ is in addition continuous, and hence is a Banach algebra morphism, as required. 
\end{proof}

\begin{proof}[Proof of Theorem~\ref{thm-m-v-for-polynomially-convex-sets}]
Form the commuting cube
\[
\xymatrix@!0@C=55pt@R=28pt{
& \varinjlim A(Z_m) \ar[dl]_-{\rho}  \ar[rr]\ar'[d][dd]
& & \varinjlim A(Z_m') \ar@{=}[dl]\ar[dd] \\
\varinjlim A(Z_{m}',Z_m'')  \ar[rr]\ar[dd]
& & \varinjlim A(Z_{m}') \ar[dd] \\
& \varinjlim A(Z_m'')\ar@{=}[dl]  \ar'[r][rr] & & \varinjlim A(Z_m'{\cap} Z_m'')\ar@{=}[dl]
\\
\varinjlim A(Z_{m}'') \ar[rr]  & & \varinjlim A(Z_{m}'{\cap} Z_{m}'') 
}
\]
with all morphisms identities or induced from restrictions.  It follows from Theorem~\ref{thm-k-theory-continuity} and  Lemma~\ref{lem-inverse-to-restriction-map} that the morphism labelled $\rho$ induces an isomorphism in $K$-theory.   In addition, it follows from Lemma~\ref{lem-limits-and-m-v}  and  Proposition~\ref{prop-application-of-inverse-function-m-v} that the front face has the Mayer-Vietoris property. Therefore it follows from Lemma~\ref{lem-mv-cube-back-to-front} that the back face does too.

Now consider a second commuting cube:
\[
\xymatrix@!0@C=55pt@R=28pt{
& \varinjlim A(Z_m) \ar[dl]  \ar[rr]\ar'[d][dd]
& & \varinjlim A(Z_m') \ar[dl]\ar[dd] \\
\varinjlim B(Z_{m})  \ar[rr]\ar[dd]
& & \varinjlim B(Z_{m}') \ar[dd] \\
& \varinjlim A(Z_m'')\ar[dl]  \ar'[r][rr] & & \varinjlim A(Z_m'{\cap} Z_m'')\ar[dl]
\\
\varinjlim B(Z_{m}'') \ar[rr]  & & \varinjlim B(Z_{m}'{\cap} Z_{m}'') 
}
\]
It follows from Theorem~\ref{thm-k-theory-continuity}  and Lemma~\ref{lem-gelfand-karoubi-a-versus-b1} that the morphisms from the back face to the front induce isomorphisms in $K$-theory.  So the front face has the Mayer-Vietoris property.  But the front face is isomorphic to   the commuting square in Theorem~\ref{thm-m-v-for-polynomially-convex-sets}.
\end{proof}

\section{Proof of the Novodvorskii  Theorem for \texorpdfstring{$B(X)$}{B(X)}}
\label{sec-mayer-vietoris-to-main-theorem}
In this section we shall prove Theorem~\ref{thm-novodvorskii-for-b-of-x},  that if $X$ is a compact and polynomially convex subset of $\C^n$, then the inclusion of  $B(X)$ into $ C(X)$ induces an isomorphism in $K$-theory.  As we already observed, this leads to a complete proof of the Novodvorskii theorem.

\begin{definition}
If $X$ is a polynomially convex compact subset of $\C^n$, then we shall denote by 
  $\DCone(X)$ the mapping cone from Definition~\ref{def-mapping-cone} for the inclusion of $B(X)$ into $C(X)$. 
\end{definition}

We shall prove Theorem~\ref{thm-novodvorskii-for-b-of-x} by showing that  $K_*(\DCone(X))=0$.

\begin{lemma}\label{lem-cone-abutting-mv-1}
Let $X\subseteq \C^n$ be a polynomially convex compact set. Let $\alpha$ be an $\R$-linear functional on $\C^n$ and   let $c \in \R$. If 
\[
X' = \{ \, z \in X : \alpha(z) \le c \, \} 
\quad \text{and} \quad 
X'' = \{ \, z\in X : \alpha(z) \ge c \, \},
\]
then the commuting square  \begin{equation*}
\xymatrix{
D(X) \ar[r] \ar[d]  & D(X') \ar[d] \\
D(X'') \ar[r]  \ar[r]& D(X' \cap X'')
}
\end{equation*} 
of restriction morphisms has the Mayer-Vietoris property.
\end{lemma}

\begin{proof}
The same diagram with either $B(X)$-type algebras or $C(X)$-type algebras has the Mayer-Vietoris property.  So the lemma follows immediately from Lemma~\ref{lem-mv-cube-mapping-cone}.
\end{proof}

\begin{corollary}\label{cor-cone-abutting-injective-from-mv}
Let $X=X'\cup X''$, as above.  If $K_*(\DCone(X' {\cap} X''))=0$, then the morphism
\[
K_*(\DCone (X)) \longrightarrow K_*(\DCone (X'))\oplus K_* (\DCone (X''))
\]
induced from restriction is injective. \qed 
\end{corollary}

 \begin{lemma}
\label{lem-continuity-of-cone-k-theory}
If $\{\, X_r: r=1,2,\dots, \}$ is a decreasing sequence of polynomially convex compact subsets of $\C^n$, then the morphism
\[
\varinjlim K_*(\DCone (X_r)) \longrightarrow K_*(\DCone (\cap_r X_r))
\]
induced from restriction to the intersection is an isomorphism.
\end{lemma}
 
\begin{proof}
 This follows from the mapping cone six-term exact sequence, the same property for the $B(X)$-algebras  and the $C(X)$-algebras, and the five lemma.
 \end{proof}

\begin{thm}
If  $X$ is any compact and  polynomially convex   subset of $\C^n$, then $K_*(\DCone(X)) {=}  0$.
\end{thm}

\begin{proof}
First, the case when $X$ is a  point is trivial.  Starting from this, we prove the theorem by induction, as follows. Suppose that $K_*(\DCone(X))=0$  for all compact polynomially convex sets   $X\subseteq \C^n$ that may be included within an $\R$-affine subspace of $\C^n$  of   dimension at most $k{-}1$, where $k {>} 0$. Now suppose that $X$ may be included in an $\R$-affine subspace $S\subseteq \C^n$ of dimension $k$, and suppose for the sake of a contradiction   that $K_*(\DCone(X)) \neq 0$. 

Choose an $\R$-linear functional $\alpha$ on $\C^n$ that is non-constant on $S$, let $a$ and $b$ be the minimal and maximal values of $\alpha$ on $X$, and let $c=(a{+}b)/2$.  Form the corresponding decomposition $X = X'\cup X''$ as in the statement of Lemma~\ref{lem-cone-abutting-mv-1} above. 

Let $x \in K_*(\DCone(X))$ be a nonzero element. By Corollary \ref{cor-cone-abutting-injective-from-mv}, we may choose one of $X'$  or $X''$, call it $X_1$,  so that the image of $x$  in  $K_*(\DCone(X_1))$  is nonzero.

Note that the difference between the maximum and minimum values of $\alpha$ on $X_1$ is $(b{-}a)/2$
By repeating the  argument starting from $X_1$,  we obtain $X_2\subseteq X_1$ such that the image of $x$ in $K_*(\DCone(X_2))$ is nonzero,  and such that the   difference between the maximum and minimum values of $\alpha$ on $X_2$ is $(b{-}a)/4$.

Continuing, there is a decreasing sequence $\{X_r\}$ of compact polynomially convex subsets of $X$  such that the image of $x$ in $K_*(\DCone(X_r))$ is nonzero and such that the   difference between the maximum and minimum values of $\alpha$ on $X_r$ is $(b{-}a)/2^r$.

But now let $Y$ be the intersection of the $X_r$.  Then $\alpha$ is constant on $Y$, so that $Y$ may be included in an $\R$-affine subspace of dimension $k{-}1$, while by Lemma~\ref{lem-continuity-of-cone-k-theory} the image of $x$ in $K_*(D(Y))$ is nonzero.  Contradiction.   
\end{proof}

\section{Appendix: A Proof of Lemma~\ref{lem-due-to-allan}}

We start from the following  result of Oka. The proof  may be found in many texts; see  for instance \cite[\S2.7]{Hormander90}.

\begin{theorem} 
Let $p$ be  a  polynomial function  on $\C^n$, and define  
\[
\mu\colon \C^n\longrightarrow  \C^{n+1} 
\]
 by $ \mu(z) = (z,p(z))$.
 Let $Y$ be a compact polynomial polyhedron  in $\C^n$, let  
 \[
 V = \{ (z,w) \in Y \times \C : |w_1|,\dots, |w_k| \le 1 \, \} ,
 \]
and let 
\[
X =  \{\,z \in Y  : |p(z)|  \leq 1\, \}  .
\]
Each  function that is holomorphic in a neighborhood of $X$ may be written in the   form 
\[
f(z) = h(\mu(z))
\]
near $X$, where $h$ is a function that is holomorphic in a neighborhood of $V$.
\end{theorem}

\begin{proof}[Proof of Lemma~\ref{lem-due-to-allan}]
Suppose  first  that $k {=} 1$.  The   function $q{=}q_1$ has nowhere vanishing gradient. So in a neighborhood of any point it may be chosen as the first coordinate  in a local coordinate system.  In that neighborhood there certainly exists a holomorphic solution to the equation $f= qh$. These local solutions patch together to define a solution throughout a neighborhood of $V$.

Now suppose that the $k{-}1$ case of the lemma has been proved (for all  $Y$).  Let $f$ be a holomorphic function near  $V$ that vanishes on the common zero set of $q_1,\dots, q_k$, as in the statement of the lemma.
Define 
\[
V' = \{\,  (z,w)\in Y' \times \C^{k-1} :    |w_1|,\dots, |w_{k-1}| \le  1 \, \bigl \} ,
\]
where 
\[
Y' = \{ \, z\in Y :  |p_k(z)| \le 1 \,\} ,
\]
and define   $f'$ in a neighborhood of  $V'\subseteq \C^{n+k-1}$ by
\[
f'(z, w_1,\dots, w_{k-1}) = f (z, w_1,\dots, w_{k-1},p_k(z)) .
\]
It vanishes on the common zero set of the polynomials 
\[
q'_j(z,w_1,\dots ,w_{k-1}) = p_j(z) - w_j\qquad ( j = 1,\dots , k{-}1) .
\]
So by the $k{-}1$ case of the lemma it may be written as  
\[
f' = q'_1h'_1 + \cdots  + q'_{k-1} h'_{k-1} ,
\]
for some    $h_1', \dots h_{k-1}'$ that are holomorphic near  $V'$.  By Oka's theorem there are    $h_1,\dots , h_{k-1}$   holomorphic near  $V$ such that 
\[
h_j (z, w_1,\dots , w_{k-1}, p_k(z)) = h'_j (z, w_1,\dots, w_{k-1})
\]
for $j=1,\dots, k{-}1$ and for $(z,w_1,\dots ,w_{k-1}) $ near $V'$. 
But then the function
\[
f -    (q_1 h_1 +  \cdots +  q_{k-1}h_{k-1} ) 
\]
is holomorphic in a neighborhood of $V$ and vanishes on the zero set of $q_k$. So by the $k {=} 1$ case of the lemma, there is a  function $h_k$  that is holomorphic in a neighborhood of  $V$ such that
\[
f -    (q_1 h_1 +  \cdots +  q_{k-1}h_{k-1} ) = q_k h_k,
\]
near $V$, as required.
\end{proof}

\bibliography{References}
\bibliographystyle{plain}

\end{document}